\documentclass[11pt,oneside]{amsart}
\usepackage{enumerate,amssymb,geometry}

\newtheorem{teo}{Theorem}
\newtheorem{lem}{Lemma}
\newtheorem*{cor}{Corollary}
\theoremstyle{definition}
\newtheorem{defi}{Definition}
\begin{document}
\author{Paul A. Blaga}
\address{University of Cluj-Napoca, Faculty of Mathematics }
\title[Morita invariance for superalgebras]{On the Morita invariance of the Hochschild homology of superalgebras}
\begin{abstract}
We provide a direct proof that the Hochschild homology of a
$\mathbb{Z}_2$-graded algebra is Morita invariant.
\end{abstract}
\maketitle
\section{Introduction}
The goal of this paper is to show that if $A$ and $B$ are two Morita equivalent unital
superalgebras, then they have the same Hochschild homology (in the $\mathbb{Z}_2$-graded sense, see
(Kassel, 1986)).
\section{The Hochschild homology of superalgebras}
The Hochschild complex for superalgebras (Kassel, 1986), is very similar to the analogous complex
for ungraded case. Namely, the chain groups are, as in the classical case, $C_m(R)=R^{\otimes
m+1}$, where, of course, the tensor product should be understood in the graded sense, while the
face maps and degeneracies are given by
\begin{equation}
\delta^m_i(a_0\otimes\dots\otimes a_m)=a_0\otimes \dots \otimes a_i a_{i+1}\otimes\dots a_n,\quad
\text{if}\quad 0\leq i <m,
\end{equation}
\begin{equation}
\delta^m_m(a_0\otimes\dots\otimes a_m)=(-1)^{|a_m|(|a_0|+\dots+|a_{m-1})}a_ma+0\otimes
a_1\otimes\dots \otimes a_{m-1},
\end{equation}
\begin{equation}
s^m_i(a_0\otimes\dots\otimes a_m)=a_0\otimes\dots\otimes a_i\otimes 1\otimes
a_{i+1}\otimes\dots\otimes a_m,\quad 0\leq i\leq m.
\end{equation}
Now the differential is defined in the usual way, meaning $d^m:C_m(R)\to C_{m-1}(R)$,
\begin{equation}
d^m=\sum\limits_{i=0}^m(-1)^i\delta_i^m.
\end{equation}
and the Hochschild  homology of the superalgebra is just the homology of the complex $(C(R),d)$. In
particular, it is easy to see that for any superalgebra $R$ we have
\begin{equation}
H_0(R)=R/\{R,R\},
\end{equation}
where $\{R,R\}$ is the subspace generated by the supercommutators.  of that element.

\section{The Morita invariance}
We shall simply give the definition of the Morita equivalence here. For a detailed approach, see
for, instance, the book of Bass (\cite{bass}). The definition is completely analogous to that from
the ungraded case.
\begin{defi}
If $A$ and $B$ are two unital, associative superalgebras over a graded
commutative superring $R$, then $A$ and $B$ are said to be \emph{Morita
equivalent} if there exists an $A-B$-bimodule $P$ and a $B-A$-bimodule  $Q$
such that $P\otimes _BQ\simeq A$ (as $A-A$-bimodules), while $Q\otimes_A
P\simeq B$ (as  $B-B$-bimodulea). The tensor products should be taken in the
graded sense.
\end{defi}
\begin{teo}\label{teorema}
Let $R$ be a commutative superring and $A$ and $B$ -- two unital
$R$-superalgebras (not necessarily commutative). Let, also, $P$ be an
$A-B$-bimodule which is projective over both rings and $Q$ -- an arbitrary
$B-A$-bimodule. Then there is an isomorphism
\begin{equation*}
F_*:H_*(A,P\otimes_B Q)\to H_*(B,Q\otimes_A P),
\end{equation*}
which is functorial in the 4-tuple $(A,B; P,Q)$.
\end{teo}
Before actually proving the theorem, let us, first, prove a technical lemma.
\begin{lem}\label{lema1}
Let $A$ be a unital, associative superalgebra over a commutative superring. If
$M$ is an arbitrary left $A$-module, while $Q$ is a \emph{projective} right
$A$-module, then
\begin{equation*}
H_n(A,M\otimes Q)=
\begin{cases}
Q\otimes_A M& \text{if} \; n=0\\
0 &\text{if} \; n\geq 1
\end{cases}.
\end{equation*}
Dually, if $N$ is a right $A$-module, while $P$ is a \emph{projective} left
$A$-module, then
\begin{equation*}
H_n(A,P\otimes N)=
\begin{cases}
N\otimes_A P& \text{if} \; n=0\\
0 &\text{if} \; n\geq 1
\end{cases}.
\end{equation*}
\end{lem}
\begin{proof}
We shall assume, first, that $Q=A$, which is, clearly, projective, when
regarded as right $A$-module. Moreover, in this case we have $A\otimes_AM\cong
M$, so what we have to prove is that
\begin{equation*}
H_n(A,M\otimes A)=
\begin{cases}
M& \text{if} \; n=0\\
0 &\text{if} \; n\geq 1
\end{cases}.
\end{equation*}
It is easily seen, however, that the standard complex for computing the
Hochschild homology of $A$ with coefficients in the module $M\otimes A$ is,
essentially, the (unnormalized) bar resolution $\beta$ of the $ M$, which has
non-vanishing homology only in degree zero and the zero degree homology is $M$.

To prove now the general case, take $Q$ an arbitrary projective right
$A$-module. Then the functor $Q\otimes_A -$ is exact and the result follows
from the isomorphism $(M\otimes Q)\otimes A^n\cong Q\otimes_A(M\otimes A\otimes
A^n)$ established by the maps
\begin{equation*}
f:(M\otimes Q)\otimes A^n\to Q\otimes_A(M\otimes A\otimes A^n),
\end{equation*}
\begin{equation*}
f((m\otimes q)\otimes(a_1\otimes \dots \otimes a_n))=(-1)^{|m||q|}
q\otimes(m\otimes 1\otimes a_1\otimes \dots \otimes a_n)
\end{equation*}
and
\begin{equation*}
g:Q\otimes_A(M\otimes A^{n+1})\to (M\otimes Q)\otimes A^n,
\end{equation*}
\begin{equation*}
g(q\otimes(m\otimes a_0\otimes \dots\otimes a_n))=(-1)^{|m||q|} (m\otimes
q)\otimes a_0a_1\otimes a_2\otimes\dots \otimes a_n.
\end{equation*}
The proof of the second part of the lemma is completely similar.
\end{proof}
\begin{proof}[Proof of the theorem~\ref{teorema}]
We consider the following family of modules and maps: $(C_{p,q},d',d'')$, where
\begin{equation*}
C_{m,n}=P\otimes B^n\otimes Q\otimes A^m,
\end{equation*}
where
\begin{equation*}
B^n=\underbrace{B\otimes B\otimes \dots \otimes B}_{n\;\text{factors}}
\end{equation*}
and
\begin{equation*}
A^m=\underbrace{A\otimes A\otimes \dots \otimes A}_{m\;\text{factors}},
\end{equation*}
and all the tensor products are considered over the ground superring $R$. Before defining
the maps $d'$ and $d''$, several remarks are in order.

First of all, it is very clear that
\begin{equation*}
C_{m,n}=C_m(A,P\otimes B^n\otimes Q),
\end{equation*}
i.e. $C_{m,n}$ is the group of the Hochschild $m$-chains of the superalgebra $A$, with
the coefficients in the  $A$-bimodule $P\otimes B^n\otimes Q$. On the other hand, up to a
cyclic permutation of the factors in the tensor product, $C_{m,n}$ is, also, the group of
the Hochschild $n$-chains of the superalgebra $B$ with coefficients in a $B-B$-bimodule.
More specifically, we have
\begin{equation*}
C_{m,n}=\omega_{m+1,n+1}\left(C_n(B, Q\otimes A^m\otimes P)\right),
\end{equation*}
where $\omega_{m+1,n+1}:Q\otimes A^m\otimes P\otimes B^n\to P\otimes B^n\otimes
Q\otimes A^m$ is the cyclic permutation of factors given by
\begin{equation*}
\begin{split}
&\omega_{m+1,n+1}(p\otimes b_1\otimes \dots \otimes b_n\otimes q\otimes a_1\otimes\dots\otimes
a_m)=\\
&=(-1)^{|p|+|q|+\sum\limits^m_{i=1}|a_i|+\sum\limits^n_{j=1}|b_j|}q\otimes a_1\otimes\dots\otimes
a_m\otimes p\otimes b_1\otimes\dots\otimes b_n.
\end{split}
\end{equation*}
Now we can use the Hochschild differentials to build the maps $d'$ and $d''$.
Let $m,n\in \mathbb{N}$ two given natural numbers. We define now, for any pair
of natural numbers, $m,n\in \mathbb{N}$, $d'_{m,n}:C_{m,n}\to C_{m-1,n}$ to be
 the Hochschild differential for $A$, with coefficients in
$P\otimes B^n\otimes Q$. Thus, on the columns we have Hochschild complexes. On
the other hand, also for any pair of natural numbers $m,n$ we define the
horizontal differentials $d''_{m,n}:C_{m,n}\to C_{m,n-1}$,
\begin{equation*}
d''_{m,n}=(-1)^mb_{m,n}\circ \omega_{m+1,n+1},
\end{equation*}
where $b_{m,n}:C_n(B,Q\otimes A^m\otimes P)\to C_{n-1}(B,Q\otimes A^m\otimes
P)$ is  the Hochschild differential. From the construction, it is obvious that
both $d'$ and $d''$ are differentials. We will prove now that they anticommute.
We have
\begin{align*}
&d''d'(p\otimes b_1\otimes\dots\otimes b_n\otimes q\otimes a_1\otimes\dots\otimes a_n)=
d''\bigg(p\otimes b_1\otimes\dots \otimes b_n\otimes qa_1\otimes a_2\otimes\dots\otimes a_m+\\ &
+\sum\limits_{i=1}^{m-1}(-1)^ip\otimes b_1\otimes\dots\otimes b_n\otimes q\otimes
a_1\otimes\dots\otimes a_ia_{i+1}\otimes\dots\otimes a_m+\\
&+(-1)^{m+|a_m|\left(|p|+|q|+\sum\limits_{j=1}^{m-1}|a_j|+\sum\limits_{j=1}^{n}|b_j|\right)}a_mp\otimes
b_1\otimes\dots\otimes b_n\otimes q\otimes a+1\otimes \dots\otimes a_{m-1} \bigg)=\\
&=(-1)^m\bigg[pb_1\otimes\dots\otimes b_n\otimes qa_1\otimes a_2\otimes\dots\otimes
a_m+\\
&+\sum\limits_{j=1}^{n-1}p\otimes b_1\otimes\dots\otimes b_jb_{j+1}\otimes \dots\otimes
b_n\otimes qa_1\otimes a_2\otimes\dots\otimes a_m+\\
&+(-1)^{n+|b_n|\left(|p|+|q|+\sum\limits_{j=1}^{m}|a_j|+\sum\limits_{j=1}^{n-1}|b_j|\right)}p\otimes
b_1\otimes\dots\otimes b_{n-1}\otimes b_nqa_1\otimes a_2\otimes\dots\otimes a_m+\\
&+\sum\limits_{i=1}^{m-1}(-1)^i\bigg(pb_1\otimes b_2\otimes\dots\otimes b_n\otimes q\otimes
a_1\otimes\dots a_ia_{i+1}\otimes \dots\otimes a_m  + \\
&+(-1)^{n+|b_n|\left(|p|+|q|+\sum\limits_{j=1}^{m}|a_j|+\sum\limits_{j=1}^{n-1}|b_j|\right)}p\otimes
b_1\otimes\dots \otimes b_{n-1}\otimes b_nq\otimes a_1\otimes\dots\otimes
a_ia_{i+1}\otimes\dots\otimes a_m \bigg)+\\
&+(-1)^{m+|a_m|\left(|p|+|q|+\sum\limits_{j=1}^{m-1}|a_j|+\sum\limits_{j=1}^{n}|b_j|\right)}\bigg(a_mpb_1\otimes
b_1\otimes\dots\otimes b_n\otimes q\otimes a_1\otimes\dots\otimes a_{m-1}+\\
&+\sum\limits_{j=1}^{n-1}(-1)^j a_mp\otimes b_1\otimes\dots b_j b_{j+1}\otimes\dots b_n\otimes
q\otimes a_1\otimes \dots \otimes a_{m-1}+\\
&+(-1)^{n+|b_n|\left(|p|+|q|+\sum\limits_{j=1}^{m}|a_j|+\sum\limits_{j=1}^{n-1}|b_j|\right)}
a_mp\otimes b_1\otimes b_{n-1}\otimes b_nq\otimes a_1\otimes\dots\otimes a_{m-1}\bigg)\bigg].
\end{align*}
On the other hand,
\begin{align*}
&d'd''(p\otimes b_1\otimes \dots b_n\otimes q\otimes a_1\otimes \dots\otimes
a_m)=(-1)^{m-1}d'\bigg(pb_1\otimes b_2\otimes\dots\otimes b_n\otimes q\otimes a_1\otimes
\dots\otimes a_m+\\ &+\sum_{i=1}^{n-1}(-1)^ip\otimes b_1\otimes\dots\otimes
b_ib_{i+1}\otimes\dots\otimes b_n\otimes q\otimes a_1\otimes\dots\otimes a_m+\\
&+(-1)^{n+|b_n|\left(|p|+|q|+\sum\limits_{j=1}^{m}|a_j|+\sum\limits_{j=1}^{n-1}|b_j|\right)}
p\otimes b_1\otimes\dots\otimes b_{n-1}\otimes b_nq\otimes a_1\otimes\dots\otimes
a_m\bigg)=
\end{align*}
\begin{align*}
&=(-1)^{m-1}\Bigg[pb_1\otimes b_2\otimes\dots\otimes b_n\otimes qa_1\otimes a_2\otimes
\dots\otimes a_m+\\ &+\sum\limits^{m-1}_{j=1}(-1)^jpb_1\otimes b_2\otimes\dots\otimes
b_n\otimes a\otimes a_1\otimes\dots \otimes a_j a_{j+1}\otimes\dots\otimes a_m+\\
&+(-1)^{m+|a_m|\left(|p|+|q|+\sum\limits_{k=1}^{m-1}|a_k|+\sum\limits_{k=1}^{n}|b_k|\right)}
a_mpb_1\otimes b_2\otimes\dots\otimes b_n\otimes q\otimes a_1\otimes\dots\otimes
a_m+\\&+\sum\limits_{i=1}^{n-1}(-1)^i\Bigg(p\otimes b_1\otimes\dots\otimes
b_ib_{i+1}\otimes\dots\otimes b_n\otimes q\otimes a_1\otimes\dots\otimes
a_m+\\&+\sum\limits_{j=1}(-1)^jp\otimes b_1\otimes\dots\otimes b_i
b_{i+1}\otimes\dots\otimes b_n\otimes q\otimes a_1\otimes\dots\otimes a_j
a_{j+1}\otimes\dots\otimes a_m+\\
&+(-1)^{m+|a_m|\left(|p|+|q|+\sum\limits_{j=1}^{m-1}|a_j|+\sum\limits_{j=1}^{n}|b_j|\right)}a_mp\otimes
b_1\otimes\dots\otimes bib_{i+1}\otimes\dots\otimes b_n\otimes q\otimes
a_1\otimes\dots\otimes a_{m-1} \Bigg)+\\ &+
(-1)^{n+|b_n|\left(|p|+|q|+\sum\limits_{j=1}^{m}|a_j|+\sum\limits_{j=1}^{n-1}|b_j|\right)}\Bigg(p\otimes
b_1\otimes\dots b_{n-1}\otimes b_nqa_1\otimes a_2\otimes\dots\otimes a_m +\\
&+\sum\limits_{j=1}^{m-1}(-1)^jp\otimes b_1\otimes\dots\otimes b_{n-1}\otimes b_nq\otimes
a_1\otimes\dots\otimes a_j a_{j+1}\otimes\dots\otimes a_m+\\
&+(-1)^{m+|a_m|\left(|p|+|q|+\sum\limits_{j=1}^{m-1}|a_j|+\sum\limits_{j=1}^{n}|b_j|\right)}a_mp\otimes
b_1\otimes\dots\otimes b_{n-1}\otimes b_nq\otimes a_1\otimes\dots\otimes a_{m-1} \Bigg)
\Bigg].
\end{align*}
An inspection shows immediately that the quantities between the square brackets in the
expressions of $d'd''$ and $d''d'$ coincide, while the signs in front of these brackets
are opposite, which means that we have
\begin{equation*}
d'd''+d''d'=0.
\end{equation*}
Thus, as we saw previously that ${d'}^2={d''}^2=0$, it follows that  the family
of modules and morphisms $(C_{m,n},d',d'')_{m,n\in\mathbb{N}}$ is a
\emph{double complex} of modules. We consider now its total complex, given, for
any $n\geq 0$, by
\begin{equation*}
Tot_n=\bigoplus_{p+q=n}C_{p,q}
\end{equation*}
and
\begin{equation*}
d_n:Tot_n\to Tot_{n-1}, \quad, d_n=\sum_{p+q=n}(d'_{p,q}+d''_{p,q}).
\end{equation*}
As it is well-known (see \cite{rotman}, from where the notations, classical, in
fact, are taken), the total complex has two canonical filtrations (a horizontal
and a vertical one) and to each of this filtration we can associate a spectral
sequence. The two spectral sequences both converge to the homology of the total
sequence. We shall show that in our case both spectral sequences collapse at
the second step. In fact, the second order terms of the two sequences are
\begin{equation*}
^IE^2_{p,q}=H'_pH''_{p,q}(C)
\end{equation*}
and
\begin{equation*}
^{II}E^2_{p,q}=H''_pH'_{q,p}(C).
\end{equation*}
In our particular case, due to the particular form of the vertical and
horizontal complexes, we get
\begin{equation}\label{proj1}
H''_{p,q}(C)=H_q(B,Q\otimes A^p\otimes P)
\end{equation}
and
\begin{equation}\label{proj2}
H'_{q,p}(C)=H_q(A,P\otimes B^p\otimes Q).
\end{equation}
As $P$ is a bimodule which is projective  at both sides, applying the previous
lemma, we can write
\begin{align*}
H''_{p,q}(C)&=
\begin{cases}
P\otimes _BQ\otimes A^p&\text{for}\;\;q=0\\
0&\text{for}\;\;q\geq 1
\end{cases}\\
H'_{q,p}(C)&=
\begin{cases}
B^p\otimes _AQ\otimes P&\text{for}\;\;q=0\\
0&\text{for}\;\;q\geq 1
\end{cases}
\end{align*}
As a consequence, we obtain for the second terms of the two spectral sequences:
\begin{align*}
^{I}E^2_{p,q}&=
\begin{cases}
H_p(A, P\otimes _BQ)&\text{for}\;\; q=0\\
0&\text{for}\;\;q\geq 1
\end{cases}
\\
^{II}E^2_{p,q}&=
\begin{cases}
H_p(B, Q\otimes _AP)&\text{for}\;\; q=0\\
0&\text{for}\;\;q\geq 1
\end{cases}
\end{align*}
Since, as we see, the two spectral sequences collapse, their limits coincide,
in fact, with the second terms. Therefore, as they should converge to the same
limit (the homology of the total complex), we have, in particular, that, for
any $n\geq 0$, we should have
\begin{equation*}
^{I}E^2_{n,0}=^{II}E^2_{n,0},
\end{equation*}
i.e.
\begin{equation*}
H_n(A, P\otimes_BQ)=H_n(B,Q\otimes_AP)
\end{equation*}
which concludes the proof (the functoriality follows from the way we
constructed the double complex).
\end{proof}
\begin{cor}
If $A$ and $B$ are Morita equivalent superalgebras, then they have isomorphic
Hochschild homologies.
\end{cor}

\end{document}